\documentclass{article}

\usepackage{amsmath}
\usepackage{amssymb}
\usepackage{amsthm}
\usepackage{graphicx}
\usepackage{amscd}
\usepackage{epic, eepic}
\usepackage{url}
\usepackage{color}

\newtheorem{theorem}{\bf Theorem}[section]
\newtheorem{lemma}[theorem]{\bf Lemma}
\newtheorem{cor}[theorem]{\bf Corollary}

\newtheorem{problem}[theorem]{\bf Problem}
\newtheorem{prop}[theorem]{\bf Proposition}

\theoremstyle{definition}
\newtheorem{nota}[theorem]{\bf Notation}

\newtheorem{remark}[theorem]{\bf Remark}
\newtheorem{defi}[theorem]{\bf Definition}

\newcommand{\cC}{\mathcal{C}}

\newcommand{\cS}{\mathcal{S}}
\newcommand{\cT}{\mathcal{T}}
\newcommand{\cH}{\mathcal{H}}
\newcommand{\cX}{\mathcal{X}}
\newcommand{\N}{\mathrm{s}}
\newcommand{\PG}{\mathrm{PG}}
\newcommand{\AG}{\mathrm{AG}}

\newcommand{\cl}{\mathrm{cl}}

\newcommand{\STS}{\mathrm{STS}}

\usepackage{todonotes}

\usepackage[utf8]{inputenc}

\title{Steiner triple systems and spreading  sets in projective  spaces}
\author{Zoltán Lóránt Nagy\thanks{ELKH--ELTE Geometric and Algebraic Combinatorics Research Group,
  E\"otv\"os Lor\'and University, Budapest, Hungary. The author is supported by the Hungarian Research Grant (NKFI) No. K 120154, 124950, 134953 and by the János Bolyai Scholarship of the Hungarian Academy of Sciences. 	E-mail: {\tt nagyzoli@caesar.elte.hu}}, Levente Szemerédi\thanks{ E\"otv\"os Lor\'and University, Budapest, Hungary. The project was  supported by the European Union, co-financed by the European Social Fund (EFOP-3.6.3-VEKOP-16-2017-00002). E-mail: {\tt szelev@caesar.elte.hu}}}
\date{}

\begin{document}

\maketitle

\begin{abstract}
  We address several extremal problems concerning the spreading property of point sets of Steiner triple systems. This property is closely related to the structure of subsystems, as a set is spreading if and only if there is no proper subsystem which contains it.
  We give  sharp upper bounds on the size of a minimal spreading set in a Steiner triple system and show that if  all the minimal spreading sets are large then the examined triple system must be a projective space. We also show that the size of a minimal spreading set is not an invariant of a Steiner triple system.
\end{abstract}

\section{Introduction}

Let $\mathbb{F}_q$  denote the Galois field with $q=p^h$ elements where $p$ is a prime, $h\geq 1$. $\mathbb{F}_{q}^d$ is the $d$-dimensional vector space over $\mathbb{F}_q$. Together with its subspaces, this structure corresponds naturally to the affine geometry $\AG(d,q)$.
The projective space $\PG(d,q)$ of dimension $d$ over $\mathbb{F}_{q}$ is  defined as the quotient 
$(\mathbb{F}_{q}^{d+1}\setminus \{\bf{0}\})/\sim$, where ${\bf a}= (a_1,\ldots,a_{d+1})\sim {\bf b}= (b_1,\ldots,b_{d+1})$ if there exists $\lambda \in \mathbb{F}_q\setminus\{\bf 0\}$ such that ${\bf a}=\lambda{\bf b}$; see \cite{KSz} for an introduction to  projective spaces over finite fields.

A block design of parameters $2-(v,k,\lambda)$ is an underlying set $V$ of cardinality $v$  together with a family of $k$-uniform subsets  whose members are satisfying the property that every pair of points in $V$ are contained in exactly $\lambda$ subsets. In this paper we deal with the case $\lambda=1$ which means the block design forms a linear space as well. If $k=3$, we get back the concept of the well known Steiner triple systems or $\STS$s in brief. Clearly, $\PG(d,2)$ and  $\AG(d,3)$ provide infinite families of Steiner triple systems. 

A triple system induced by a proper subset $V' \subset V$ consists of those triples whose elements do not contain any element of $V \setminus V'$. A nontrivial Steiner subsystem of $\mathcal{S}$ is an $\STS(n')$ induced by a proper subset $V'\subset V$, with $|V'|=n'>3$. Speaking about a  triple system's subsystem, we always suppose that it is of order greater than $3$. Similarly but not analogously, we  call a subset $V'\subset V$ of the underlying set of a triple system $\mathcal{F}$ \emph{nontrivial} if it has size at least $3$ and it is not an element of the triple system. \\
In contrast with the above mentioned geometries where Steiner subsystems are abundant, most $\STS$s do not have any nontrivial subsystem. In \cite{NagyZ} the first author examined the so called {\em spreading property} of triple systems and defined a certain neighbourhood concept which describes how far a triple system can  be from containing subsystems. This paper is devoted to analyse the other side of the spectrum using these concepts to explore $\STS$s which are relatively close to geometric settings such as affine and projective spaces. Besides, it turns out that related extremal problems are in connection with  geometric objects of independent interest such as saturating sets or covering codes.

\begin{defi}[Closure and spreading property, \cite{NagyZ}]\label{szomszed} Consider a graph $G=G(V,E)$ that admits a triangle decomposition. This decomposition corresponds to a  linear triple system $\mathcal{F}$. For an arbitrary set $V' \subset V$, $N(V')$ denotes the set of  its \emph{neighbours}: $$z\in N(V') \Leftrightarrow z\in V \setminus V' \mbox{ \ and \ } \exists xy \in E(G[V']): \{x,y,z\}\in \mathcal{F}.$$
The \emph{closure} $\cl_{\mathcal{F}}(V')$ of a subset $V'$ w.r.t. a (linear) triple system $\mathcal{F}$ is the smallest set $W \supseteq V'$ for which $|N(W)|=0$ holds. If it does not make any confusion, we omit the index and briefly note it by $\cl(V')$.  Note that the closure uniquely exists for each set $V'$.\\ We call a vertex subset $V'$ \emph{spreading} if  $\cl(V')=V$,
and we also call a (linear) triple system $\mathcal{F}$ \emph{spreading} if $\cl(V')=V$ for every nontrivial subset $V'\subset V$.
\end{defi}

Consequently,   a $\STS(n)$  is subsystem-free if and only if $|N(V')|>0$ holds for all nontrivial subsets $V'$ of  the underlying set  $V$ of the system. Note that  Doyen  used the term  \textit{non-degenerate plane} for $\STS$s with the spreading property  \cite{Doyen, Doyen2}.

A Steiner triple system which is subsystem free or which contains only few subsystems contains  sets $V'$ of only $3$ vertices which spreads, i.e. for which $\cl(V')=V$.
This observation gives raise to the following question.

\begin{problem}\label{minsp}
What is the minimum size of a spreading set which exists in every $\STS(n)$ of order $n$?
\end{problem}

In connection with this problem, we proved the result below which describes the Steiner triple systems attaining the maximum of the minimum size.

\begin{theorem}\label{maxofmin} In any Steiner triple system $\STS(n)$ of order $n>1$, there exists a spreading set $U$ of size $|U|\leq \log_2(n+1)$  and this bound  is best possible for infinitely many values of $n$.
\end{theorem}

\begin{theorem}\label{unicity} Suppose that the smallest size of a spreading set in $\STS(n)$   is $\log_2(n+1)$. Then $\STS(n)$ is a projective space over $\mathbb{F}_2$.
\end{theorem}

One would be tempted to conjecture that a much stronger stability result also holds, i.e., that in any Steiner triple system $\STS(n)$ of order $n>1$ which is not isomorphic to some projective space $\PG(d,2)$, there exists a spreading set $U$ of size $|U|\leq \log_3(n)+1$  with  equality  attained for affine spaces $\AG(d,3)$ for $n=3^d$.

However, this is far from the truth. We call a spreading set minimal if none of its proper subsets is a spreading set.

\begin{prop}\label{almostmax} There exists a family of Steiner triple systems $\STS(n)$  with minimal spreading set of size $|U|= \log_2(n+1)-1$.
\end{prop}

We also investigate whether the minimal spreading sets of an arbitrarily chosen $\STS$ are of the same size or how much their cardinality may alter. Although the minimal spreading sets of  $\STS$s mentioned above, i.e. of  geometric $\STS$s or of $\STS$s without subsystems have a uniform cardinality, this does not hold for every $\STS$.

\begin{theorem}\label{minsize_cardins}
There exists a Steiner triple system which has both a minimal spreading set of size $3$ and $n$.
\end{theorem}

Next we make connections to previously studied concepts. Let us define recursively the {\em quasi-closure of a set $S$ after $i$ spreading steps}.\\ $S_0:=S$, and $S_i=S_{i-1}\cup N(S_{i-1}).$ Obviously, if $S_i=S_{i+1}$ for some $i$, then $\cl(S)=S_i$. 


\begin{defi}\label{quasi}
Suppose we are given a subset $S\subset V$ of the underlying set of a Steiner triple system. If the quasi-closure $S_1$  of $S$ after $1$ step coincides with the complete set $V$, then $S$ is called a saturating set. 
\end{defi}

Saturating sets are already defined in projective geometries and has a large literature, see  e.g. \cite{Davydov, Lins, Giu-survey}. In the terminology of projective geometries, these are point sets  $U$ such that the lines determined by the pairs of points from $U$ covers every point of the projective geometry. These objects have a significant application in coding theory as well.
 In fact, saturating sets in projective
geometry  are corresponding to covering codes, more precisely, linear codes with covering radius $2$  see \cite{DGMP2, DGMP3, 10, 26, 27}. 
The above Definition \ref{quasi} basically  generalizes  the concept of saturation to partial linear spaces. 
We get back to this relation in the last section of the paper.



The paper is organised as follows. In the next preliminary section we collect some basic results that we intend to apply later on. We discuss the results on the size of minimal spreading sets in Section 3, and prove Theorem \ref{maxofmin} \ref{unicity} and \ref{almostmax}. Section 4 is dedicated to the description of  Steiner triple system constructions which contain minimal spreading sets of different sizes. Finally we pose some open problems in Section 5. 

\section{Preliminaries}

In this preliminary section we point out the correspondence between the set of closures and the set of subspaces in projective and affine geometries and address some partial results concerning  the deviation of the size of intersections from the average intersection size of a point set $U$ and a hyperplane $\cH$ in projective spaces.

\begin{lemma}[folklore]\label{folk}
If $U$ is a point set of $\PG(n,q)$  such that every line intersects $U$ in $0,1 $ or $q+1$ points then $U$ is a subspace of  $\PG(n,q)$. If $U$ is a point set of $\AG(n,q)$  such that every line intersects $U$ in $0,1 $ or $q$ points then $U$ is a subspace of $\AG(n,q)$ provided that $q>2$.
\end{lemma}

\begin{prop} The closure $\cl(U)$ of an arbitrary set $U \in \PG(n,2)$ coincides with the smallest subspace in which $U$ can be embedded. 
\end{prop}

\begin{proof} First observe that $\cl(U)$  is a subset of the smallest subspace in which $U$ can be embedded. Any point $P$ of $\cl(U)$  is obtained after some number of spreading steps, i.e., there is a number $i$ s.t. $P$ is contained in the quasi-closure $U_i$ of a set $U$ after $i$ spreading steps, but $P\not\in U_{i-1}$. This also means that there is a line $\ell$ which contains $P$ and intersects $U_{i-1}$ in exactly $2$ points. Therefore by induction each quasi-closure $U_i$ is contained in the  smallest subspace in which $U$ can be embedded, since each point of $U_i\setminus U_{i-1}$ is a point of a line fully contained in the subspace in view.\\
On the other hand, the closure intersects each line of the space in $0, 1$ or $3$ points, hence it must be a subspace according to Lemma \ref{folk}. 
\end{proof}

Next we examine the distribution of the size of intersections of a fixed set having $m$ elements and the hyperplanes of the projective plane. We would apply the result for the case when the order is $q=2$, but it is easy to generalize the method. 

\begin{prop}\label{szoras}
Let $U$ be a set of $m$ points in $\PG(n,2)$ and let   $u(\cH)$ denote the  size of the intersection of $U$ and a hyperplane $\cH$ of the projective space. Then there exists a hyperplane $\cH$ for which $$|u(\cH)-m/2|> \sqrt{\frac{m}{4}-\frac{m^2}{2^{n+3}}}.$$ Here note that the expected value of the intersection size is $\frac{2^{n}-1}{2^{n+1}-1}m$ which is almost $m/2$, thus the above estimate is closely related to the variance.
\end{prop}

\begin{proof}
  Let us consider the sum  $\sum_{\cH}|u(\cH)-m/2|^2 $ for all hyperplanes $\cH$ of $\PG(n,2)$. Taking into consideration that the number of hyperplanes is $2^{n+1}-1$, the number of hyperplanes through a point is $2^{n}-1$ and the number of hyperplanes through a line is $2^{n-1}-1$, we obtain the following  by standard double counting.
  
  $$\sum_{\cH}|u(\cH)-m/2|^2= 2\binom{m}{2}(2^{n-1}-1)-(m-1)m\cdot(2^n-1)+\frac{m^2}{4}(2^{n+1}-1)=-\frac{m^2}{4}+m\cdot2^{n-1}.$$
  The claim thus in turn follows.
\end{proof}

\section{Size of spreading sets}

Throughout the rest of the paper, we use capital letters for ordinary vertex sets and calligraphic  capital letters for sets together the (induced) partial Steiner system defined on the set. 

Building on the observations 
we made in the Preliminaries, we give sharp upper bounds on the maximum size of minimal spreading sets of Steiner triple systems $\STS(n)$.
First we prove Theorem \ref{maxofmin} by confirming the bound $|U|\leq \log_2(n+1)$.

\begin{proof}[Proof of Theorem \ref{maxofmin}]
We can construct a spreading set in the following way. If $n=3$ then the statement clearly holds. Otherwise, we do the following procedure: $U_1=\{v_0, v_1\}$  for an arbitrary point pair, and $U_{i+1}=U_{i}\cup v_{i+1}$ where $v_{i+1}\in V \setminus \cl(U_i)$ is arbitrarily chosen if $V \setminus \cl(U_i)$ is nonempty, otherwise $U=U_i$. The key observation is the following: $|\cl(U_{i+1}|\geq 2|\cl(U_{i}|+1$ which is the easy part of the Theorem of Doyen and Wilson \cite{DW}, claiming that there exists an $\STS(v)$ which contains an
$\STS(w)$ as a proper subdesign if and only if $v \geq 2w + 1$,  and  $v, w \equiv  1 { \mbox \ or \ } 3  \pmod 6$.
This in turn implies the upper bound of $|U|$.

As it was pointed out in Theorem \ref{unicity}, we have equality if $n=2^{d+1}-1$ and the triple system is isomorphic to $\PG(d,2)$.
\end{proof}

\begin{cor}
If there is a minimal spreading set $U$ in a Steiner triple system $\STS(n)$ then $|U|\leq\log_2(n+1)$.
\end{cor}

\begin{remark}
This is sharp for $\AG(2,3)$ as all of the minimal spreading sets are of cardinality $3=\lfloor\log_2(9+1)\rfloor$. But if we restrict ourselves to equality in Theorem \ref{maxofmin}, then we one can obtain a characterization of Steiner triple systems for which the bound is attained. These are exactly the projective spaces.
\end{remark}


\begin{proof}[Proof of Theorem \ref{unicity}]
In this proof we will denote the underlying set of $\STS(n)$ by $X$. Furtermore, we will call a subset $U\subseteq X$ a quasi-subspace if $\cl(U)=U$ and use  the notation $U\leq X$. Our goal is confirm that the structure and  'dimension' of quasi-subspaces is in one-to-one correspondence with the respective structure of subspaces of a projective space of order $\log_2(n+1)-1.$ Here we follow the description of projective spaces from \cite{KSz}. 

Each subset $S\subset X$ constructed in the proof of Theorem \ref{maxofmin} is a minimal spreading set. Indeed, otherwise there there would be a $W\subset S$ minimal spreading set of size $|W|<\log_2(n+1)$.

Observe that is $Y,Z\leq X$ are  two quasi-subspaces then $Y\cap Z$ is also a quasi-subspace, as its closure should be both in $Y$ and $Z$.

Consider a quasi-subspace $Y\leq X$ and two minimal spreading sets $A,B\subset Y$ of $Y$. Then, applying Theorem \ref{maxofmin}, we obtain $|A|=|B|=\log_2(|Y|+1)$  by induction. Let us denote by  $\dim Y:=|A|-1$ the concept analogue to dimension. To complete the proof  we have to prove the dimension theorem.

Let $Y,Z\leq X$ be two quasi-subspaces.  Then it is easy to see that $Y$ and $Z$ are disjoint if and only if $\dim(\cl(Y\cup Z))=\dim Y+\dim Z+1$.

If $Y,Z\leq X$ are two arbitrary quasi-subspaces then $\dim(\cl(Y\cup Z))+\dim(Y\cap Z)=\dim Y+\dim Z$: take a minimal spreading set $A$ of $Y\cap Z$ and a set $B$ from $Z$ such that $A\cup B$ is a minimal spreading set in $Z$. Then $Y$ and $\cl(B)$ are disjoint so we can apply the statement above.
\end{proof}

We continue by studying whether some stability holds for this bound. The next statement asserts that there is no stability in the sense that if we decrease the size by $1$, we will have infinitely many Steiner triple systems $\STS(n)$ having minimal spreading sets with this respective size. 

\begin{prop}\label{almost}
There exists a family of Steiner triple systems $\STS(n)$ with minimal spreading set of size $|U|=\log_2(n+1)-1$.
\end{prop}

\begin{proof}
Our goal is present a construction gained from a projective space by slightly modifying the system triples (i.e., lines) of $\PG(d,2)$ with $n=2^{d+1}-1$ so that the structure of subspaces do not change significantly.

Consider $\PG(d,2)$ with a minimal spreading set $\{v_0,\dots, v_d\}$. Take a basis in $\mathbb F_2^{d+1}$ such that the $i$th basis vector is the representative vector of $v_i$ ($i\in\{ 0,\dots,d\}$).

Doyen \cite{Doyen} showed that there are Steiner triple systems of  any admissible order larger than $9$ which  contain no nontrivial subsystems, i.e., every minimal spreading system is of size $3$. Let us modify the $3$-dimensional subspace generated by $v_0,v_1,v_2,v_3$ in $\PG(d,2)$ to a Steiner system described above of size $15$. This modified 'subspace' is now generated by any $3$ 'non-collinear' element. We can choose the position of the triples such that $v_1,v_2,v_3$ the triplet on each pair $\{v_1, v_2\}$, $\{v_2, v_3\}$, $\{v_1, v_3\}$ in this $\STS$ remain intact.

We will show that this modified triple system that we denote  by $\cX$ has a minimal spreading set of size $|U|=\log_2(n+1)-1=d$. Note that the points of $\PG(d,2)$ and $\cX$ are the same so we will denote them with their homogeneous coordinates.

Consider $U=\{v_1,\dots, v_d\}$. This is a spreading system in ${\cX}$ because $\{v_0,\dots, v_d\}$ was a spreading system in $\PG(d,2)$ and $\cl_{\cX}(v_1,v_2,v_3)=\cl_{\PG(d,2)}(v_0,v_1,v_2,v_3)$. To show that $U$ is a minimal spreading set it is enough to prove that any subset of $U$ of size $d-1$ is not a spreading set.

$U\setminus\{v_d\}$ is not a spreading system because $\cl_{\cX}(U\setminus\{v_d\})=\cl_{\PG(d,2)}(v_0,\dots v_{d-1})$ which is not the whole system since $\{v_0,\dots, v_d\}$ is a minimal spreading set in $\PG(d,2)$. Similarly for any $k\in\{ 4,\dots, d\}$ the set $U\setminus\{ v_k\}$ is not a spreading system of $\cX$.

$U\setminus\{ v_l\}$ (for $l\in\{ 1,2,3\}$) is also not a spreading system  because
$\cl_{\PG(d,2)}(U\setminus\{ v_l\})$ is a subspace where all the triplets are remained intact during the modification, thus $\cl_{\PG(d,2)}(U\setminus\{ v_l\})$
 coincides with the same point set.
\end{proof}

\section{Minimal spreading sets of different sizes}

In this section we will show that for an arbitrary integer $n>3$ there exists a Steiner triple system $\cS_n$ which has a minimal spreading set of size $3$ and one of size $n$.

Let $a_1,a_2,\ldots,a_n$ be an affine base in $\AG(n-1,3)$. Let us denote by $\cH_i$ the affine hyperplanes $\cl(a_1,\ldots,a_{i-1},a_{i+1},\ldots,a_n)$ for $1\leq i\leq n$. Finally,  let $\cX:=\bigcup_{i=1}^n\cH_i$ and  let $X_i$ be the underlying set of $\cH_i$.

\begin{prop}
The system $\cX$ is a partial Steiner triple system and has a minimal spreading set of size $n$.
\end{prop}

\begin{proof}
The affine space $\AG(n-1,3)$ is a Steiner triple system and $\cX\subset \AG(n-1,3)$ thus $\cX$ is a partial Steiner triple system.

The set $\{a_1,\ldots,a_n\}$ is a minimal spreading set of $\cX$ as $\cl_{\cX}(\{a_1,\ldots,a_n\})=\cX$ and $a_i\not\in\cH_i=\cl(a_1,\ldots,a_{i-1},a_{i+1},\ldots,a_n)$.
\end{proof}

Let us choose points $b_1$, $b_2$ and $b_3$  from $X_1\setminus(X_2\cup\ldots \cup X_n)$, $X_2\setminus(X_1\cup X_3\cup\ldots \cup X_n)$ and $X_3\setminus(X_1\cup X_2\cup X_4\cup\ldots \cup X_n)$, respectively. Note that these sets are nonempty. Add $n+4$ further points $b_{4},\ldots,b_{n+7}$ to $X$. For $k=1,\ldots, n+4$ add the triplet $(b_k,b_{k+1},b_{k+3})$ and for $l=4,\ldots,n+3$ add the triplet $(b_l,b_{l+4},a_{l-3})$ to the existing partial triple system $\cX$. Let us denote this extended triple system by $\cX'$.

It is easy to see that $\cX'$ is a partial Steiner triple system i.e. there does not exist a point pair $x,y\in X'$ such that there is more than one triplet in $\cX'$ which contains both $x$ and $y$.

\begin{prop}
The set $\{b_1,b_2,b_3\}$ is a minimal spreading set of $\cX'$.
\end{prop}

\begin{proof}
From the construction we get that $a_l,b_k\in\cl_{\cX'}(\{b_1,b_2,b_3\})$ for all $l=1,\ldots, n$ and $k=1,\ldots,n+7$. Thus $\cl_{\cX'}(\{a_1,\ldots,a_n\})\subset\cl_{\cX'}(\{b_1,b_2,b_3\})$. But $\cX\subset\cX'$ thus $X=\cl_{\cX}(\{a_1,\ldots,a_n\})\subset \cl_{\cX'}(\{a_1,\ldots,a_n\})$. So $\cl_{\cX'}(\{b_1,b_2,b_3\})=X'$ which means that the set $\{b_1,b_2,b_3\}$ is a spreading set. It must be also a minimal spreading set because the closure of a set of size $2$ in a partial Steiner triple system is either of size $2$ or of size $3$.
\end{proof}

\begin{prop}
The set $\{a_1,\ldots,a_n\}$ is a minimal spreading set of $\cX'$.
\end{prop}

\begin{proof}
$\{a_1,\ldots,a_n\}$ is a spreading set because its closure contains the set $\{b_1,b_2,b_3\}$ which is a spreading set. It is also a minimal spreading set. 
Indeed, $\cl_{\cX'}(\{a_1,\ldots,a_{i-1},a_{i+1},\ldots,a_n\})=\cl_{\cX'}(X_i)=\cl_{\cX}(X_i)=X_i\neq X'$ for all $i$, since none of the triplets added  to $\cX$ has more than $1$ point from $X_i$. 
\end{proof}

We will use a theorem of Bryant and Horsley to show that not only a partial Steiner triple system exists with the required property but also a Steiner triple system.

\begin{theorem}[Bryant, Horsley \cite{Linderc}]
  Any partial Steiner triple system of order $u$ can be embedded in a Steiner triple system of order $v$ if $v\equiv1,3\pmod 6$ and  $v\geq 2u+1$.
\end{theorem}

The theorem above guarantees that there exists a Steiner triple system which contains $\cX'$. Let $\cS$ be an STS which contains $\cX'$ and is of the least possible size.

\begin{theorem}
  There exists a Steiner triple system which has both a minimal spreading set of size $3$ and $n$.
\end{theorem}

\begin{proof}
We will show that $\cS$ is appropriate.

Both the sets $\{a_1,\ldots,a_n\}$ and $\{b_1,b_2,b_3\}$ are spreading sets in $\cS$ because their closure is $X'$ in $\cX'$ so in $\cS$ their closure is the same as of $\cX'$ but $\cl_{\cS}(X')=S$ because $\cS$ is an STS contains $\cX'$ of  minimal size.

The set $\{b_1,b_2,b_3\}$ is a minimal spreading set because its each subset has a closure of size at most $3$.

The set $\{a_1,a_2,\ldots,a_n\}$ is also a minimal spreading set. Let $A\subset\{a_1,a_2,\ldots,a_n\}$. The closure $\cl_{\cS}(A)$ of $A$ is the same as the closure $\cl_{\cX}(A)$ because it is a sub-STS since it was an affine subspace in the original affine space thus during the extension of $\cX$ we could not add any triplet which has at least two points in $\cl_{\cX}(A)$.
\end{proof}



\section{Concluding remarks and open problems}

We begin this section with some connections to saturating sets of projective spaces.
\begin{nota} The minimum size of a saturating set in $\PG(n,q)$ is denoted by $\N(n,q)$.
\end{nota}

The trivial lower bound on $\N(n,q)$ is due to Lunelli and Sce, and reads as follows \cite{Ughi}.

\begin{prop}
$$(q-1)\binom{\N(n,q)}{2}+\N(n,q)\geq \frac{q^{n+1}-1}{q-1},$$
thus $\N(n,q)> 2q^{\frac{n-1}{2}}$.
\end{prop}

For smaller values of $q$, the bound  $N(n,q)> 2q^{\frac{n-1}{2}}$ can be refined. The bound in turn implies the next statement.

\begin{cor}\label{lower_sat}
 $$\N(n,2)\geq \sqrt{2^{n+2}-2}-0.5, \ \ \ \N(n,3)\geq \sqrt{\frac{3^{n+1}-1}{2}}.$$
\end{cor}

There are no known matching or asymptotically matching upper bounds in general. We refer  the interested reader to the recent paper \cite{Lins} concerning small constructions.

We point out a method to improve slightly the lower bound.

The core of the idea relies on an extremal combinatorial problem which is interesting on its own.
\begin{problem}\label{elteres}
Let us take a point set $U$ in a projective space $PG(n,q)$ of cardinality $m$. Determine $$\max_{|U|=m} \min_{H \subset PG(n,q) \mbox{\ is \ a \ hyperplane}} |U\cap H|$$ and  $$\min_{|U|=m} \max_{H \subset PG(n,q) \mbox{\ is \ a \ hyperplane}}|U\cap H|.$$
\end{problem}

Concerning Problem \ref{elteres}, Proposition \ref{szoras} implies that  there exists a hyperplane  which intersects the set $U$ in either significantly less point than expected, or in significantly more points. Note that in general we cannot expect to obtain similar bounds for intersections with many point or intersections with few points separately. Indeed, if $U$ is the point set of a hyperplane or a point set of the complement of a hyperplane of $\PG(n,2)$, then the intersection sizes are $|U|$ or  $\frac{|U|-1}{2}$ and $0$ or  $\frac{|U|+1}{2}$, respectively, hence the minimum, resp. maximum size only differs from expected number by less than $1$.

Note that these questions are strongly connected to determine the minimum size of a $t$-fold blocking set in projective space $\PG(n,q)$. Once one have a lower bound on a $t$-fold blocking set which exceeds $|U|=m$, we get that there should be a hyperplane on which there are at most $t-1$ points of $U$. On results concerning $t$-fold blocking sets of projective spaces we refer to \cite{Barat, Sziklai}.

To present the link to saturating sets in projective spaces, consider the case $q=2$. (If $q>2$ one can follow the same lines but the calculation is slightly more involved.) 
Take a  saturating set $S$ of $m$ points and hyperplane on which the number $m_1$ of points $v\in S$ is either small or large compared to $m/2$. Each point not on the hyperplane should be saturated, hence it is either a point of $S$ or on a line determined by two point from $S$: one from the hyperplane and one from its complement. Thus we obtain a  necessary condition $m-m_1+m_1(m-m_1)\geq 2^n-1$.
Once we have a result stating that for each point set of size $m$ we can find a hyperplane containing $m_1$ points where $|m/2-m_1|$ is large, we in turn get a lower bound on $m$. An instance for an application is the result of Proposition \ref{szoras}, which provides a slightly better bound than the trivial one from Corollary \ref{lower_sat}.

Finally we mention an open problem in connection with Proposition \ref{almost}. As we have seen  before, large minimal spreading sets could also be constructed via the perturbation of the structure of a projective geometry over $\mathbb{F}_2$, however all these triple systems are defined on a point set of size one less than a power of $2$. 

\begin{problem} Does there exist a Steiner triple system $\cS$ of order $n$ such that the minimal spreading set in $\cS$ is larger than $\log_3(n)+1$ and the order of the system is not of form $n= 2^t-1$?
\end{problem}


\end{document}